\documentclass[10pt,journal,letterpaper,compsoc,onecolumn]{IEEEtran}

\ifCLASSOPTIONcompsoc
\else
\fi
\ifCLASSINFOpdf
\else
\fi
\usepackage[cmex10]{amsmath}
\usepackage{amssymb}
\usepackage{epsfig}

\newtheorem{theorem}{Theorem}
\newtheorem{corollary}[theorem]{Corollary}
\newtheorem{definition}[theorem]{Definition}

\def \ipp {\mu}

\begin{document}
%
\title{Performance Analysis of Spectral Clustering on Compressed, Incomplete and Inaccurate Measurements}
%
%
%
%

\author{Blake~Hunter~
        and~Thomas~Strohmer
\IEEEcompsocitemizethanks{\IEEEcompsocthanksitem B. Hunter and T. Strohmer are with the Department of Mathematics, University of California, Davis,
CA, 95616.\protect\\
E-mail: blakehunter,strohmer@math.ucdavis.edu}

\thanks{Partially supported by NSF DMS grant   1042939, 0811169 and 0636297.}}

\IEEEcompsoctitleabstractindextext{%
\begin{abstract}
Spectral clustering is one of the most widely used techniques for extracting the underlying global structure of a data set.  Compressed sensing and matrix completion have emerged as prevailing methods for efficiently recovering sparse and partially observed signals respectively.  We combine the distance preserving measurements of compressed sensing and matrix completion with the power of robust spectral clustering.  Our analysis provides rigorous bounds on how small errors in the affinity matrix can affect the spectral coordinates and clusterability.  This work generalizes the current perturbation results of two-class spectral clustering to incorporate multi-class clustering with $k$ eigenvectors.  We thoroughly track how small perturbation from using compressed sensing and matrix completion affect the affinity matrix and in succession the spectral coordinates.  These perturbation results for multi-class clustering require an eigengap between the $k^{th}$ and $(k+1)^{th}$ eigenvalues of the affinity matrix, which naturally occurs in data with $k$ well-defined clusters.  Our theoretical guarantees are complemented with numerical results along with a number of examples of the unsupervised organization and clustering of image data.
\end{abstract}

}

\maketitle


\IEEEdisplaynotcompsoctitleabstractindextext

%
\IEEEpeerreviewmaketitle

\section{Introduction}
\IEEEPARstart{D}{ata} mining has become one of the fastest growing research topics in mathematics and computer science. Spectral clustering is a tool for extracting meaningful information from data by grouping similar objects together \cite{ESL01}.   The method uses the eigenvector of an adjacency matrix for embedding the data into a space that captures the underlying group structure \cite{NCIS}.   High-dimensional signals, magnetic resonance images, and hyperspectral images can be costly to acquire; even simple direct comparisons could be infeasible among such data sets.  Our work shows that the meaningful organization extracted from spectral clustering is preserved under the perturbation from making compressed, incomplete and inaccurate measurements.   Using bounds on the perturbation of eigenvectors, we establish error bounds of the spectral embedding when matrix completion and compressed sensing measurements are used.  Given some error $N \epsilon$ in the entries of an affinity matrix $A \in \mathbb{R}^{N \times N}$, we show that the space spanned by the first $k$ eigenvector are all within $O(N \epsilon)$ of the span of the unperturbed eigenvectors.  We prove that the perturbed spectral coordinates are within $O(N \epsilon)$ of a unitary transform of the unperturbed coordinates and can give k-means cluster assignments within $O(N \epsilon)$ of the unperturbed case.   This analysis holds true when the error perturbation in the entries of an affinity matrix $| A(i,j) - \tilde{A}(i,j) | \leq \epsilon$ is caused from making compressed sensing measurements, matrix completion or any other process, making our perturbed clustering results widely applicable.  This work shows that spectral clustering is achievable in the compressed domain and with missing and noisy entries as long as the spectral gap is satisfied.  
 
As the dimensionality of data increases, data mining tasks such as clustering and classification can become intractable or costly to obtain.  Traditional clustering algorithms must perform dimensionality reduction to make the problem tractable before they can be applied.  Learning in the compressed domain was first proved possible using support vector machines in \cite{CLSDR}.  In addition to techniques for exact recovery of sparse signals, compressed sensing provides a bound on the error derived from making random measurements \cite{DDCS,SSRIIM}.   We show how errors from using compressed sensing can affect the affinity matrix and in turn the spectral coordinates. 

In practice, data may be missing, lost or not fully observed.   There are numerous tasks where you are given only a small portion of the data in hope to understand the entire set.  A set of high dimensional images containing $k$ hidden subsets with missing entries is impossible to cluster with standard clustering methods alone due to the lack of information from the incomplete data.  Suppose the set of images are stacked as rows of a matrix, if this matrix is low rank then the set of similar images contain a wealth information about the missing entries of any one of the images.   Matrix completion uses the information from similar rows of the data matrix to fill in the missing entries making clustering possible.  Matrix completion is an emerging area of research that provides efficient algorithms to reconstruct the full matrix $X$ from a small subset of observed entries via nuclear-norm minimization \cite{emcco09, pcrnomc09}.  Since data matrices are usually not exactly low rank, the matrix completion procedure results in errors in the recovered entries. We analyze how these errors propagate through the spectral clustering steps and derive rigorous bounds under which clusterability of the data is preserved after matrix completion. An exampled is depicted in Fig.~\ref{fig:misclass3facesMS4}, where face images can be successfully clustered even if only 5\% of the image data are available. For details we refer to Section~\ref{sec:res}. 

\begin{figure}[!h]
	\begin{center}
		\includegraphics[width=\textwidth]{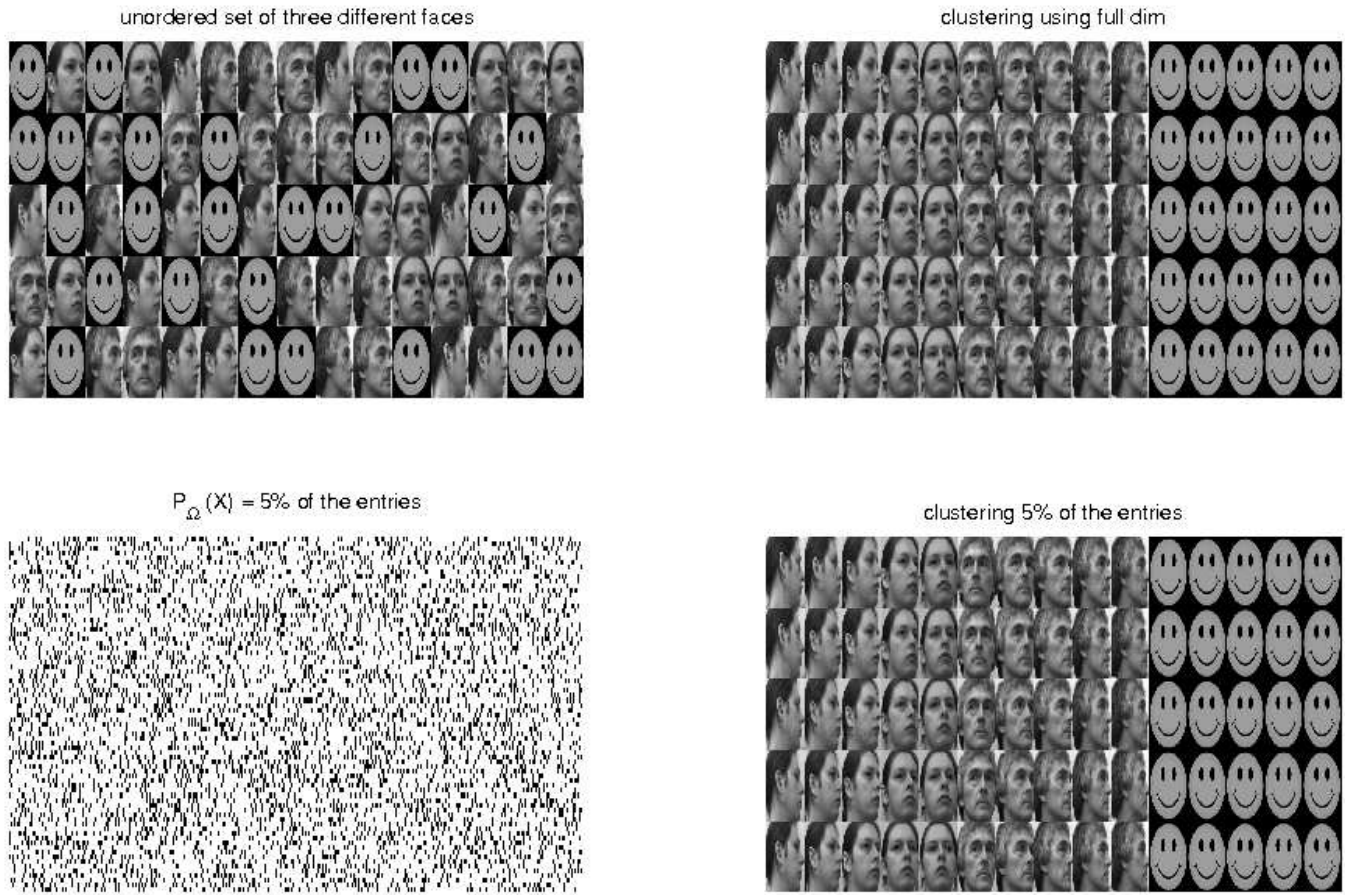}
	\end{center}
	\caption {\footnotesize Clustering a dataset of
100 images of three different people's faces in a range of poses from profile to frontal views where only 5\% of each image is observed.  Only two of the three people's face can be published so a happy face is used here in place of the third person's face for display purposes only.  The images are clustered by applying the matrix completion perturbed spectral clustering coordinates
($2^{nd}$ and $3^{rd}$ eigenvectors). } \label{fig:misclass3facesMS4}
\end{figure}

The structure of the paper is as follows.  Previous results on perturbation of eigenvectors and spectral clustering with perturbed data are presented in Section~\ref{sec:back}.   Our robustness analysis is presented in Section~\ref{sec:method}.  It is followed by a theoretical justification and a comparison to other methods.  Small perturbation error in the affinity matrix is shown to be well behaved in the spectral coordinates in Section~\ref{sec:perteig}.   Measurement errors due to using matrix completion and compressed sensing measurements are shown to give small error in the affinity matrix and in turn the spectral coordinates in Section~\ref{sec:measure} where bounds on the span of the first $k$ eigenvectors under small perturbations are provided. Section~\ref{sec:res} is dedicated to numerical results, where the method is applied tp both synthetic and real world image data sets.

\section{Background}\label{sec:back}
\subsection{Spectral Clustering}
	Clustering is an unsupervised learning problem that reveals the underlying structure from unlabeled data.  The goal of clustering is to partition objects into groups such that objects within the same group are similar. Standard clustering such as k-means requires the space in which the objects are represented, to be linearly separable.  Spectral clustering methods detect non-convex patterns and linearly non-separable clusters.  This allows for a wider range of underlying geometries, making them more flexible \cite{ESL01, KKSC}.

Standard spectral clustering uses the eigenvectors of the graph formed by local distances between data points, to reveal the global structure of the data set.  
A traditional choice of edge weights  uses the Gaussian kernel,
\begin{equation} 
W(x_i,x_j)=\exp \left(  - \frac{\|x_i- x_j\|_2^2}{2 \sigma}\right).\label{eq:defW}
\end{equation}

A random walk on the graph is defined by normalizing the rows of $W$ to give the stochastic matrix, 
\[P=D^{-1}W,\]
where \[D_{i,i} = \sum_{k} W(x_i,x_k)\] is a diagonal matrix of row sums of $W$. 

The original symmetry of $W$, lost in $P=D^{-1}W$, can be preserved by defining $A$ as, 
\begin{equation}
	A=D^{-\frac{1}{2}} W D^{-\frac{1}{2}}. \label{eq:defA}
\end{equation}
Spectral clustering finds the top $k$ eigenvectors $V_k\in\mathbb{R}^{N\times k}$ of $A$ to provide coordinates for clustering.   To cluster the original data a standard clustering algorithm like k-means is then applied to the rows of $V_k$, as illustrated in \cite{ESL01}.    

To show why the eigenvectors of spectral clustering works, Shi and Malik proved in \cite{NCIS} that the second eigenvector of $P$ is the real valued solution to minimizing the normalized cut problem, that bipartitions the points of the graph.  The graph bipartitioning problem was extended to multi-class clustering by using multiple eigenvectors as described in \cite{MCSC,SR4KM}.

\subsection{Perturbation of the second eigenvector}
Earlier results have shown that spectral clustering using the second eigenvector is robust to small perturbation of the data, see \cite{FASC,scwpd}.  These results are based on the following perturbation theorem by Stewart \cite{IMC}.   

\begin{theorem}
  Let $\tilde{A}=A+E$ be a perturbation of $A$ and let $\lambda_i$ and $v_i$ be the $i^{th}$ eigenvalue and eigenvector of $A$ and $\tilde{v_i}$ be the $i^{th}$ eigenvector of $\tilde{A}$ respectively, then
 \begin{equation} \|\tilde{v_2} - v_2  \| \leq \frac{1}{\lambda_2-\lambda_3}\| E \| +O\left( \| E \|^2 \right). \end{equation}
  \label{thm:Stewart}
  \end{theorem}

This holds provided the gap between the second and third eigenvalue is not close to zero.  This is not the case of data sets with more than two underlying clusters.  The number of eigenvalues close to one is equal to the number of separate clusters.  Consider the simplest example, where there are three single points forming three non-connected clusters.

The affinity matrix $A$ will be the $3\times 3$ identity and will have three eigenvalues of one, hence the gap between the second and third eigenvalue is zero.  
 In general for data sets with $k$ well separated clusters there will be $k$ eigenvalues close to $1$, making the eigengap $\lambda_2-\lambda_3$ close to zero, and destroying the bound of the second eigenvector  in Theorem~\ref{thm:Stewart}. 
 
  Small perturbations in the entries of an affinity matrix can lead to large perturbation in the eigenvectors.  First consider the ideal clustering data set with two underlying clusters where each data point has equal similarity to each intraclass point and is dissimilar to each interclass point. The ordered point would partition the affinity matrix $A$ into a block diagonal matrix,
\[ A=\left( 
\begin{matrix}
	 A_1 & 0 \\ 
	0 & A_2
\end{matrix} \right) \] 
where $A_i$ is a matrix of all ones and $0$ is the zero matrix.  Let 
\[\tilde{A} = A+E\] 
be a perturbation of that matrix where $E$ has entries uniformly distributed from $0$ to $\epsilon$.  Previous analysis used in \cite{FASC} says that when $\epsilon$ is small, the second eigenvector of $A$ and $\tilde{A}$ are close, i.e. satisfy Theorem~\ref{thm:Stewart}.  The second eigenvector is a positive/negative indicator of each point's cluster membership.  The vector is constant for all the points within each cluster, but there is an arbitrary choice of which sign is assigned to each cluster, as can be seen in Figure~\ref{fig:eigenex2c}.  With the correct choice made then the bound in Theorem~\ref{thm:Stewart} holds.  The Euclidean distance, $\|\tilde{v}_2 - v_2\|$, can be large when the sign is chosen incorrectly, but what is preserved, is the space spanned by $\{v_2\}$ and $\{ \tilde{v}_2 \}$.  We define closeness of these subspaces using canonical angles. 

\begin{definition}
Let $\mathcal{V}_k$ and $\tilde{\mathcal{V}}_k$ be subspaces spanned by the orthonormal eigenvectors $v_i, \ldots , v_{i+k}$ and $\tilde{v}_i, \ldots , \tilde{v}_{i+k}$. And let $\gamma_1\leq \ldots \leq \gamma_k$ be the singular values of $\left[ v _i \cdots  v_{i+k}\right]^T \left[ \tilde{v}_i \cdots \tilde{v}_{i+k} \right]$.  Then the values,
\[\theta_i = \cos^{-1} \gamma_i\] 
are called the \textbf{canonical angles} between $\mathcal{V}_k$ and $\tilde{\mathcal{V}}_k$.  \label{def:CanonAng}
\end{definition}
Define $\mathcal{V}_k$ and $\tilde{\mathcal{V}}_k$ to be close if the largest canonical angle, $\theta_1$, is small.  See \cite{REBP70,MPTSS} for more details.    

\begin{figure}[ht]     
	\begin{center}
		\includegraphics[width=\textwidth]{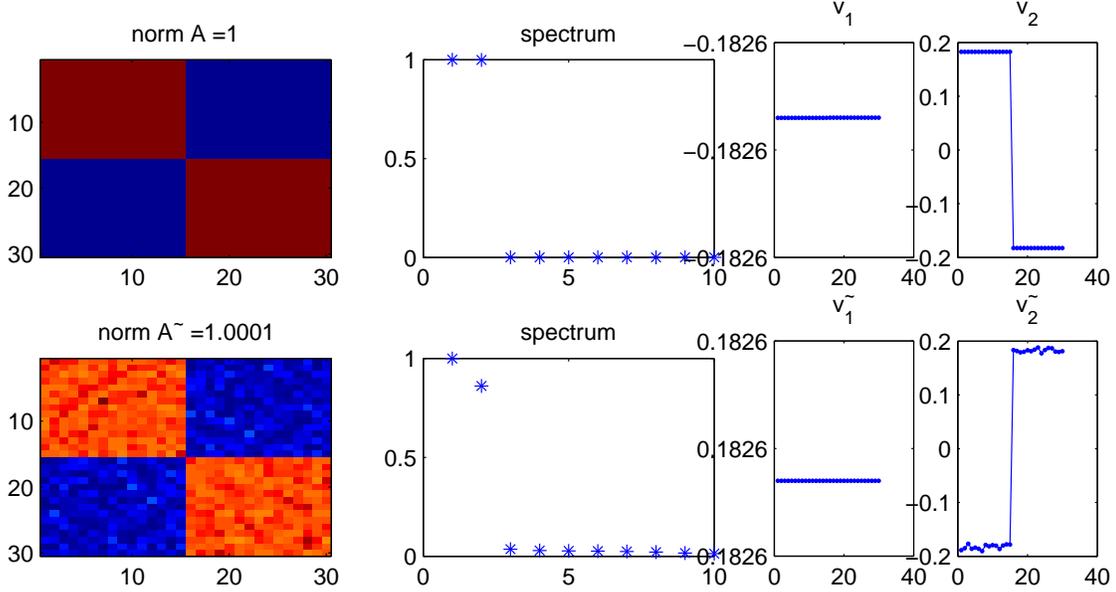}
	\end{center}
	\caption {\footnotesize   30 data points with two underlying clusters.  \textbf{Top:}(from right to left) The affinity matrix $A$, the first ten eigenvalues, the first eigenvector $v_1$ and the second eigenvector $v_2$.  \textbf{Bottom:}(from right to left) The affinity matrix $\tilde{A}$, the first ten eigenvalues, the first eigenvector $\tilde{v}_1$ and the second eigenvector $\tilde{v}_2$.  The canonical angle between $v_2$  and $\tilde{v}_2$ is $\theta_1=.0199 $. }\label{fig:eigenex2c}
\end{figure}

Now consider a data set with three underlying clusters, so
\[ A=\left( 
\begin{matrix}
	 A_1 & 0  & 0\\ 
	 0 & A_2 & 0\\
	0 & 0 & A_3
\end{matrix} \right) .\] 
Here the eigengap between the second and third eigenvalue is small which destroys the bound in Theorem~\ref{thm:Stewart}.  Previous work \cite{scwpd} argues that even though this bound fails, in practice the second eigenvectors can still give the correct coordinates for clustering but provide no justification.  Figure~\ref{fig:eigenex3c} shows that the measure of clusterability of the perturbed spectral coordinates is not captured by the difference between eigenvectors, $\|\tilde{v}_2 - v_2\|$ but by the canonical angle between the subspaces spanned by the eigenvectors.  Even though the Euclidean distance, $\|\tilde{v}_2 - v_2\|$ is large, the clusterability of $v_2$ is maintained in $\tilde{v}_2$.  This robustness of the clusterability can be characterized by the small canonical angle between  $\{v_2,v_3\}$  and  $\{\tilde{v}_2,\tilde{v}_3\}$.

\begin{figure}[ht]     
	\begin{center}
		\includegraphics[width=\textwidth]{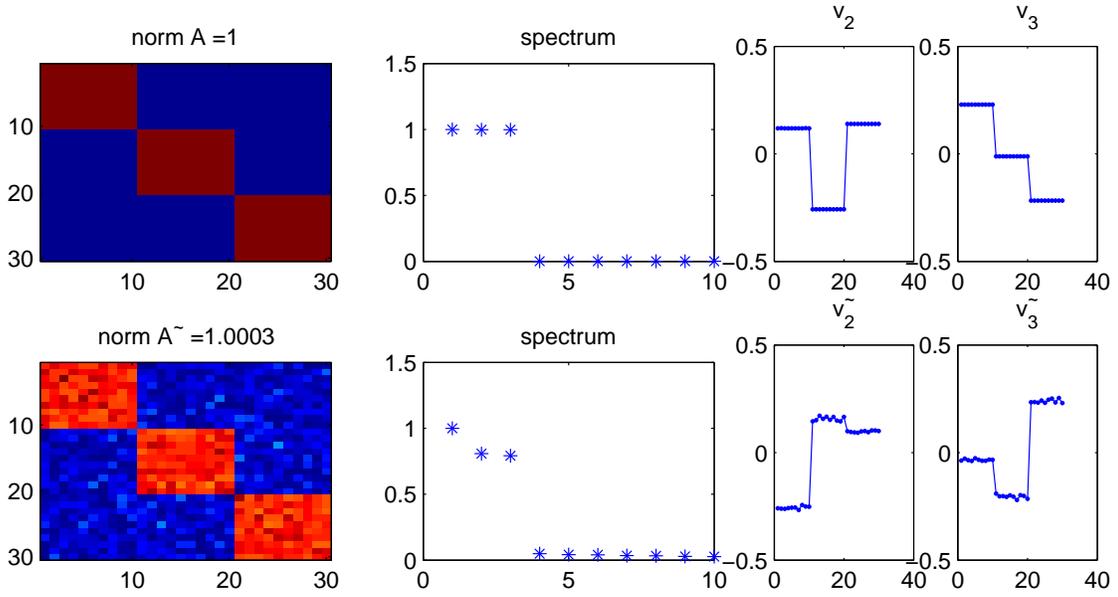}
	\end{center}
	\caption {\footnotesize   30 data points with three underlying clusters.  \textbf{Top:}(from right to left) The affinity matrix $A$, the first ten eigenvalues, the second eigenvector $v_2$ and the third eigenvector $v_3$.  \textbf{Bottom:}(from right to left) The affinity matrix $\tilde{A}$, the first ten eigenvalues, the second eigenvector $\tilde{v}_2$ and the third eigenvector $\tilde{v}_3$.  The normed difference $\|\tilde{v}_2 - v_2\| = 1.7738$ is large where as the largest canonical angle between the column space of $\{v_2,v_3\}$  and  $\{\tilde{v}_2,\tilde{v}_3\}$ is $\theta_1=.0548 $ is still small. }\label{fig:eigenex3c}
\end{figure}

\section{Spectral Clustering on Perturbed Data}\label{sec:method}

Standard spectral clustering methods use the first $k$ eigenvectors of $A$, constructed from local distances between points to provide a $k$ low-dimensional representation of the data, which is used as coordinates for clustering.  Matrix completion and compressed sensing measurements are guaranteed to give a good approximation of these Euclidean distances even when only a small fraction of the entries are observed or the number of measurements is much less than the ambient dimension.  Our method merges the distance preserving dimensionality reduction of compressed sensing and Matrix completion with the power of spectral clustering.  Previous results on spectral clustering on perturbed data are based on perturbation bounds when using the second eigenvector for bipartitioning.  Our analysis generalizes these results to incorporate multi-class clustering using the top $k$ eigenvectors.

Assume that the local distances required for standard spectral clustering are replaced by perturbed distances.  
Define the local distance using perturbed data $X$ as,
\begin{equation} \tilde{d}(x_i, x_j)=\|\tilde{x}_i - \tilde{x}_j\|_2. \label{eq:deftd} \end{equation}
Construct a graph with edge weights
\begin{equation} \tilde{W}(x_i,x_j)=\exp \left(  - \frac{\|\tilde{x}_i- \tilde{x}_j\|_2^2}{2 \sigma}\right).\label{eq:deftW}\end{equation}
Define the symmetric $N \times N$ matrix 
\begin{equation} \tilde{A}=\tilde{D}^{-\frac{1}{2}}\tilde{W}\tilde{D}^{-\frac{1}{2}} \label{eq:deftA} \end{equation}  
where $\tilde{D}_i,_i = \sum_{k=1}^N \tilde{W}(x_i,x_k)$.

The first $k$ eigenvectors $\tilde{V}_k\in\mathbb{R}^{N\times k}$  of $\tilde{A}$ are used as a $k$ dimensional representation of the data.  With these spectral coordinates preserved, k-means is applied to the rows of $\tilde{V}_k$, to cluster the original data points $x_i$.  

For classification with partially labeled data, the membership of an object is matched to that of its neighbors by performing k-means in the eigenvector domain.  We quantify the error in misclassified data by defining the misclassification rate as,
\[\rho = \frac{1}{N} \sum_{i=1}^{N} \chi_{\{ I_i \neq \tilde{I}_i \}}, \]
where $\chi$ is the indicator function, $I_i$ is the value indicating the class membership of $x_i$ and $\tilde{I}_i$ of $\tilde{x}_i$.

\begin{figure}[ht]
	\begin{center}
$\begin{array}{c c c}
\includegraphics[width=\textwidth]{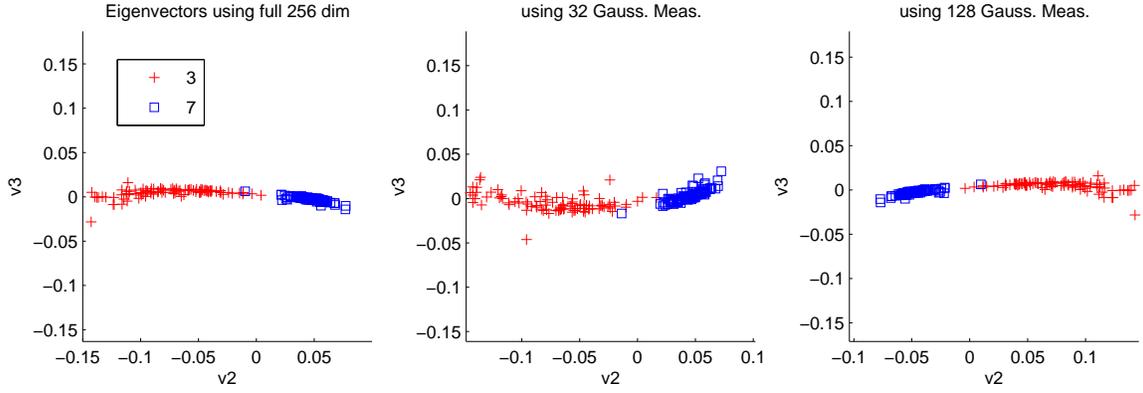}
\end{array}$
	\end{center}
	\caption {\footnotesize   The handwritten digits $\{1,3\}$ data set is projected onto the 2nd and 3rd eigenvectors.  \textbf{Left:} From the graph formed using Euclidean distances between points. \textbf{Center:} Using distances from 30 random Gaussian measurements. \textbf{Right:} 128 random Gaussian measurements.\label{fig:dig_1_3_eig_sort} }
\end{figure}

Often when analyzing high dimensional signals, the underlying structure of interest only has a few degrees of freedom or is sparse in some unknown basis.  We analyze two types of perturbed data $\tilde{X}$ in Section \ref{sec:measure}.   We show that instead of requiring the local distances be made in the large ambient dimension, measurements can be made on the order of the dimension of the hidden underlying point cloud structure.  Using the controllable error from taking compressed sensing measurements ($\tilde{d}(x_i, x_j)=\|\Phi x_i - \Phi x_j\|_2$) and matrix completion ($\hat{d}(x_i, x_j)=\|\hat{x}_i - \hat{x}_j\|_2$), we establish perturbation bounds of the affinity matrix, the eigenvectors, the spectral coordinates and the clustering memberships.

Assume that there is an underlying $s$-sparse representation $y_i$ of the data $x_i$, where $y_i=B x_i$ is a known or unknown unitary transformation of $x_i$.  Let $\Phi$ be a random  $m\times n$ matrix, with Gaussian $\mathcal{N}(0,1)$ entries.
Define the local distance in \eqref{eq:deftd} using $m$ compressed sensing measurements as,
\begin{equation}  \tilde{x}_i=\Phi x_i. \label{eq:deftdCS}\end{equation}
Local distances are preserved under the perturbation of making compressed sensing measurements when the number of measurements $m$, is large enough. 

In many applications such as collaborative filtering, computer vision and wireless sensor networks the data being analyzed maybe lost, damaged and only partially observed.  Matrix completion reconstructs a low-rank data matrix from a small subset of its entries. 

Now assume $X$ is a low rank matrix, under some constraints on the matrix known as the strong incoherence property $X$ can be reconstructed from a fraction of the entries $P_{\Omega}(X)$. 
Define the local distance in \eqref{eq:deftd} using the reconstructed matrix $\hat{X}$, where
\begin{equation*} \tilde{x}_i= \text{ row $i$ of }  \hat{X}. \end{equation*}
Local distances are preserved under the perturbation of matrix completion when $X$ obeys the strong incoherence property.

Our analysis provides rigorous bounds on how small errors in the affinity matrix can affect the spectral coordinates and clusterability.  Our analysis not only applies to compressive spectral clustering but generalizes the current results of spectral clustering on perturbed data to incorporate multi-class clustering with $k$ eigenvectors.  We show perturbation due to compressed measurements and matrix completion, preserve the affinity matrices i.e. for any $0<\epsilon<1$, given the number of measurements or observed entires large enough  then $|A_{i,j}-\tilde{A}_{i,j}|\leq \epsilon$. With this we show the span of the first $k$ eigenvectors of $A$ is close to the span of $k$ eigenvectors $\tilde{A}$, $\| \sin{\theta}\|_F \leq \frac{N \epsilon}{\alpha}$.  We then show given that the matrices $\| A - \tilde{A}  \|_F  \leq N \epsilon$, the perturbed spectral coordinate are within $O(N \epsilon)$ of a unitary transform $Q$ of the unperturbed coordinates, $\|\tilde{v}(i)-v(i) Q\|_2 \leq (1+\sqrt{2})\frac{N \epsilon}{\alpha}$.  When spectral clustering is preformed in the compressed domain or after applying matrix completion, the eigenvectors of $\tilde{A}$ or $\hat{A}$ can replace the eigenvectors of $A$ as coordinates for clustering and classification as seen in Figure~\ref{fig:dig_1_3_eig_sort} and \ref{fig:misclass3facesMS4}.

\section{Robustness of clustering under perturbation of the top $k$ eigenvectors}\label{sec:perteig} 
Previous results in approximate spectral clustering based on the perturbation of the second eigenvector \cite{scwpd}, are limited to bipartitioning and require assumptions on the distributions of the perturbation of the components of $v_2$.  We expand the theory of approximate spectral clustering to partitioning data with $k$ underlying clusters by showing what is preserved under small perturbations is that the column space of the first $k$ eigenvectors of $A$.  When the eigengap between $\lambda_1, \lambda_2, \ldots,\lambda_k$ is small then the column space spanned by their corresponding eigenvectors will be close to the column space spanned by the eigenvectors of the perturbed matrix.  

Let 
\[ V_k=\left[ 
\begin{matrix}
\ & \ & \ & \ \\
v_1 &  v_2 & \ldots & v_k\\
\ & \ & \ & \  \\
\end{matrix} \right] \]
where $v_l$ is the column eigenvector corresponding to the $l^{th}$ largest eigenvalue of $A$.  Similarly define
\[ \tilde{V}_k=\left[ 
\begin{matrix}
\ & \ & \ & \ \\
\tilde{v_1} &  \tilde{v_2} & \ldots & \tilde{v_k}\\
\ & \ & \ & \  \\
\end{matrix} \right] \] for $\tilde{v_l}$ the $l^{th}$ eigenvector of $\tilde{A}$. 

 Given there is an eigengap,  $(\lambda_k - \lambda_{k+1})$, between the first $k$ eigenvalues $\Sigma_k=\text{diag}(\lambda_2, \lambda_2, \ldots , \lambda_k)$ and the last $N-k$ eigenvalues, $\Sigma_{N-k}=\text{diag}(\lambda_{k+1}, \ldots , \lambda_N)$ of $A$, we write the block decomposition of the eigenvectors as $V=[V_k,V_{N-k}]$ where  $V_{N-k}=[v_{k+1}, \ldots , v_N ]$.  So the eigendecomposition can be written as,
\[ \left( V_k V_{N-k} \right)^H  A \left( V_k V_{N-k} \right)= \left(
\begin{matrix}
	\Sigma_k & 0 \\ 
	 0 & \Sigma_{N-k}
\end{matrix} \right), \]
and similarly for $\tilde{A}$,
\[ \left( \tilde{V}_k \tilde{V}_{N-k} \right)^H  \tilde{A} \left( \tilde{V}_k \tilde{V}_{N-k} \right)= \left(
\begin{matrix}
	\tilde{\Sigma}_k & 0 \\ 
	 0 & \tilde{\Sigma}_{N-k} 
\end{matrix} \right). \]

\begin{theorem}
 Let $\lambda_i$, $v_i$, $\tilde{\lambda}_i$, $\tilde{v_i}$ be the $i^{th}$ eigenvalues and eigenvectors of $A$ and $\tilde{A}$ respectively, and let $\Theta=\text{diag}(\theta_1,\theta_2, \ldots, \theta_k)$ be the diagonal matrix of canonical angles between the column space of $V_k=[v_1,v_2, \ldots , v_k ]$ and $\tilde{V}_k =[\tilde{v}_1,\tilde{v}_2, \ldots , \tilde{v}_k ]$.  If there is a gap $\alpha > 0$ such that
\[ | \tilde{\lambda}_k - \lambda_{k+1}| \geq \alpha\] 
and
\[  \tilde{\lambda}_k  \geq \alpha\] 
then
\[ \| \sin{\Theta}\|_F \leq \frac{1}{\alpha}\| A \tilde{V}_k -\tilde{V}_k \tilde{\Sigma}_k \|_F \]
where $\sin\Theta$ is taken entrywise. \label{Thm:WedinEigen}
\end{theorem}

This is a reformulation of the $\sin \Theta$ Theorem of Davis and Kahan \cite{REBP70}.

Note that $\lambda_k$ and $\tilde{\lambda}_k$ are close by the Mirsky Theorem, so the first condition requires $\alpha$ to be less than the eigengap between $\lambda_k$ and $\lambda_{k+1}$. The second condition requires $\alpha$ to be less than the eigenvalues in the first block $\tilde{\Sigma}_k$.   This extended the perturbation results of the second eigenvector in Theorem~\ref{thm:Stewart}, to bounding the perturbation of the column space of the first $k$ eigenvectors. 

\begin{corollary}\label{thm:proj}
Let $P_{V_k}$ and $P_{\tilde{V}_k}$ be the orthogonal projection on to $V_k$ and $\tilde{V}_k$. If there is a $\alpha>0$ such that $\lambda_k -\lambda_{k+1}\geq \alpha$ and $\lambda_k\geq \alpha$, then,  
\begin{equation}
\| P_{V_k} - P_{\tilde{V}_k}\|_F \leq \frac{\sqrt{2}}{\alpha}\| A - \tilde{A} \|_F.\label{eq:project}
\end{equation} 
\end{corollary}

\begin{proof}
Assume there exists a $\alpha>0$ such that $\lambda_k -\lambda_{k+1}\geq \alpha$ and $\lambda_k\geq \alpha$.  With the eigengap $\lambda_k -\lambda_{k+1}\geq \alpha$ and the block decomposition of $V$ and $\tilde{V}$, Theorem~\ref{Thm:WedinEigen} bounds the canonical angles $\theta$, between the column space of $\{v_1,v_2,...,v_k\}$ and $\{\tilde{v}_1,\tilde{v}_2,...,\tilde{v}_k\}$: 
 \[ \| \sin{\theta}\|_F \leq \frac{1}{\alpha} \| (\tilde{A} - E) \tilde{V}_k1 -\tilde{V}_k \tilde{\Sigma}_k \|_F \leq \left( \frac{1}{\alpha} \| \tilde{A} \tilde{V}_k -\tilde{V}_k \tilde{\Sigma}_k \|_F + \| E\tilde{V}_k\|_F \right) \leq \frac{1}{\alpha}\| A - \tilde{A} \|_F. \]
 It is shown in \cite{MPTSS} that the norm from canonical angles and the norm from projections satisfy, 
\begin{equation*}
\| P_{V_k} - P_{\tilde{V}_k}\|_F = \sqrt{2}\|\sin \Theta\|_F.
\end{equation*}
 Combining these gives the result.
\end{proof}

This establishes closeness between the space spanned by the first $k$ eigenvectors of $A$ and the first $k$ eigenvectors of $\tilde{A}$, bounding the difference between the low dimensional embedding of projecting onto the first $k$ eigenvectors of $A$ and $\tilde{A}$.  Euclidean geometry is essentially preserved if the eigengap is satisfied.  This is shown is Theorem~\ref{thm:VkQ} and Corollary~\ref{cor:Vk(i)}


\begin{theorem}\label{thm:VkQ}
Let $V_k$ be the matrix formed by the top $k$ column eigenvectors of $A$ and $\tilde{V}_k$, the matrix formed by the top $k$ eigenvectors of $\tilde{A}$ defined above. If $\lambda_k -\lambda_{k+1}\geq \alpha$ and $\lambda_k\geq \alpha$ then there is a unitary matrix  $Q$ such that 
\[\|\tilde{V}_k-V_k Q\|_2 \leq (1+\sqrt{2})\frac{1}{\alpha}\| A - \tilde{A} \|_F.\]
\end{theorem}

\begin{proof}

To compare $V_k$ with $\tilde{V}_k$ we find an unitary matrix $Q$ such that $ \| \tilde{V}_k -V_kQ\|_F $ is minimum.  It is shown in \cite{mcgvl} that this can be found by taking the singular value decomposition of $V_k^T\tilde{V}_k$, 
\[Y^T V_k^T\tilde{V}_k Z = \text{diag}(\cos \theta_i),\]
where $\theta_i$ are the canonical angles between the column space of $V_k$ and the column space of $\tilde{V}_k$.  Thus $Q=Y Z^T$ is the orthogonal matrix that minimizes $ \|\tilde{V}_k  -V_kQ\|_F $. 

\begin{align*}
\|\tilde{V}_k-V_k Q\|_2 &\leq \|\tilde{V}_k-V_k V_k^{T}\tilde{V}_k\|_2 + \|V_kV_k^{T}\tilde{V}_k-V_k Q\|_2\\
&\leq \|P_{\tilde{V}_k}\tilde{V}_k-V_k V_k^{T}\tilde{V}_k\|_2 + \|V_k\|_2\|V_k^{T}\tilde{V}_k- Q\|_2\\
&\leq \|P_{\tilde{V}_k}-P_{V_k}\|_2 \|\tilde{V}_k\|_2 + \|V_k\|_2\|V_k^{T}\tilde{V}_k- Q\|_2\\
&= \|P_{\tilde{V}_k}-P_{V_k}\| _2 + \|V_k^{T}\tilde{V}_k- YZ^T\|_2 \ \ \ \ \ \text{ ($\|\tilde{V}_k\|_2=1=\|V_k\|_2$)}\\
&= \|P_{\tilde{V}_k}-P_{V_k}\|_2  + \|Y( \cos \Theta) Z^T -YZ^T\|_2\\
&= \|P_{\tilde{V}_k}-P_{V_k}\|_2  + \|Y\|_2\| \cos \Theta - I \|_2\| Z^T\|_2\\
&= \|P_{\tilde{V}_k}-P_{V_k}\|_2  + \| \cos \Theta - I \|_2\ \ \ \ \ \text{ ($Y$ and $Z$ are unitary)}\\
&\leq \|P_{\tilde{V}_k}-P_{V_k}\|_F  + \| \cos \Theta - I \|_F\\
\end{align*}

Since, $\cos \theta_i$ are singular values, $\cos \theta_i$ are positive for all $i$, which means,
\[1+ 2 \cos^2 \theta_i\leq 1+ 2 \cos\theta_i \]
\[1 -2\cos\theta_i+ \cos^2\theta_i  \leq 1- \cos^2\theta_i\]
\begin{equation}
(\cos\theta_i - 1)^2 \leq \sin^2 \theta_i.\label{eq:cosleqsin}
\end{equation}

So we have,
\begin{align*}
\| \cos \Theta - I \|_F&=\sqrt{\sum_{i=1}^{k}(\cos\theta_i-1)^2}\\
&\leq\sqrt{\sum_{i=1}^{k}(\sin^2\theta_i)} \ \ \text{by \eqref{eq:cosleqsin}}\\
&=\|\sin\Theta \|_F\\
\end{align*}
Applying \eqref{eq:project} to
\[\|\tilde{V}_k-V_k Q\|_2 \leq \|P_{\tilde{V}_k}-P_{V_k}\|_F + \|\sin \Theta\|_F\]
gives the result.
\end{proof}

Each data point can be clustered by applying a clustering algorithm, such as k-means or PCA, to the rows of $V_k$ or $\tilde{V}_k$, see \cite{ESL01,OSCAA}.  The rows of $V_k$ and $\tilde{V}_k$ both provide coordinates for clustering and classification.  To analyze the clusterability we need to compare the rows of $V_k$ and $\tilde{V}_k$.  Corollary~\ref{cor:Vk(i)} shows that the spectral coordinates' Euclidean geometry can preserved under perturbation if the eigen gap is satisfied.  This generalizes the current perturbation results of bipartitioning data by thresholding the second eigenvector to incorporate k-way clustering with k eigenvectors.    

 \begin{corollary}\label{cor:Vk(i)}
 Let $v(i)$ be the $i^{th}$ row of the matrix $V_k$ formed by the top $k$ column eigenvectors of $A$ and $\tilde{v}(i)$ be the $i^{th}$ row of $\tilde{V}_k$ formed by the top $k$ eigenvectors of $\tilde{A}$.  If $\lambda_k -\lambda_{k+1}\geq \alpha$ and $\lambda_k\geq \alpha$ then
 \[\|\tilde{v}(i)-v(i) Q\|_2 \leq (1+\sqrt{2})\frac{1 }{\alpha}\| A - \tilde{A} \|_F\] 
 where $Q$ is the orthogonal matrix that minimizes $ \|\tilde{V}_k  -V_kQ\|_F $.
\end{corollary}

\begin{proof} 
\[\|\tilde{v}(i)-v(i) Q\|_2\leq \|\tilde{V}-V Q\|_2 \leq (1+\sqrt{2})\frac{1 }{\alpha}\| A - \tilde{A} \|_F\] 
\end{proof}

\section{Perturbation from Compressible, Incomplete and Inaccurate Measurements}\label{sec:measure}
The theory that we have developed shows how small perturbations in the affinity matrix affect the spectral coordinates.   This analysis is based on having knowledge of the amount error in $\|A-\tilde{A}\|_F$.  Direct knowledge of the error in the affinity matrix arises naturally in many areas such as wireless sensor networks, where information can be lost or corrupted between independent spatially distributed sensors.  More commonly in practice this error is not known directly, but what is observed is the error in the data it self $X-\tilde{X}$.  We analyze how error developed from taking compressible, incomplete and inaccurate measurements affects the affinity matrix and the spectral coordinates.  

\subsection{Perturbation of the entries of A from taking compressed sensing measurements}
Traditionally, spectral clustering methods use local Euclidean distance, $d(x_i, x_j)=\|x_i - x_j\|_2$ to create the affinity matrix $A$.  We show that $\tilde{A}$ defined using compressed sensing measurements can be made arbitrarily close to $A$. We prove that $\tilde{V}_k$ can be made arbitrarily close to a unitary transform of $V_k$ and the standard spectral coordinates can be replaced by the compressed spectral coordinates to provide the same clustering assignments.

\subsection{Compressed Sensing Background}		
Dimensionality reduction or low dimensional representation has become a central problem in signal and image processing.  As the dimensionality of data increases, data mining tasks such as clustering and classification can become intractable or costly to obtain.  Traditional clustering algorithms must perform dimensionality reduction to make the problem tractable before they can be applied.  Usually some type of transform is required to be computed to produce a sparse vector to make clustering feasible.  In \cite{CLSDR} it was shown that learning using support vector machines can be done in the measurement domain without having to transform the data to a sparse representation.  We provide detail of how and how well spectral clustering works in the measurement domain through careful analysis. 

Compressed sensing provides techniques for exact recovery of sparse signals $x$ from random measurements $z=\Phi x$, where $\Phi$ is a random $m\times n$ matrix. In general this is an ill-posed problem, but the assumption of sparsity makes recovery possible, \cite{DDCS,SSRIIM,CTDLP05}.  A major result of compressed sensing proves that exact recovery of sparse signals can be guaranteed when the number of measurements $m=O(s\log n/s )$, which is much less than the ambient dimension $n$, see \cite{SSRIIM,DDCS} and the references therein.  A central idea of compressed sensing is the restricted isometry property. 

	\begin{definition}
The \textbf{restricted isometry property} (\textbf{RIP}) holds with parameters $(r,\delta)$ where $\delta \in (0,1)$ if
\begin{equation} 
(1-\delta) \| x\|_2 \leq \| \Phi x \|_2 \leq (1+\delta) \| x \|_2.
\label{eq:RIP}
\end{equation}
holds for all $s$-sparse vectors $x$.
\end{definition}

	It has been shown that random Gaussian, Bernoulli, and partial Fourier matrices satisfy the RIP with high probability~\cite{RV08}.

There is a large body of work that uses random projections for dimensionality reduction.  Most of these methods are shown to work well in practice but have no theoretical uniform guarantees.  Compressed sensing provides numerical bounds on the error produced when taking random measurements with uniform guarantees.

\begin{theorem} 
 Let $W_{i,j}=e^{ - \frac{\|x_i - x_j\|_2^2}{2 \sigma}}$ and $A= D^{-1/2} W D^{-1/2}$ where $D_{i,i}= \sum_{k=1}^N W(x_i,x_k)$.  And let $\tilde{A}=\tilde{D}^{-\frac{1}{2}}\tilde{W}\tilde{D}^{-\frac{1}{2}}$ where $\tilde{D}_{i,i}= \sum_{k=1}^N \tilde{W}(x_i,x_k)$. If $\tilde{W}_{i,j}=e^{ - \frac{\| \Phi x_i - \Phi x_j\|_2^2}{2 \sigma}}$ where the $x_i$s are $s$-sparse and $\Phi$ satisfies the RIP with 
 \[\delta = \frac{\epsilon}{4 \max_{i,j}\left\{ \frac{\|x_i - x_j\|_2^2}{2 \sigma}\right\}},\]
then for $0<\epsilon <1$, 
 \begin{equation*} | \tilde{A}_{i,j} -  A_{i,j} | \leq \epsilon.\end{equation*} \label{thm:AijCS}
\end{theorem}
\begin{proof}

By the quadratic form of the restricted isometry property (RIP) from \cite{RV08} we have,
\begin{equation} 
(1-\delta) \| x_i-x_j\|_2^2 \leq \| \Phi x_i-\Phi x_j \|_2^2 \leq (1+\delta) \| x_i-x_j \|_2^2.
\label{eq:qRIP}
\end{equation}
Multiplying by $\frac{-1}{2\sigma}$ and exponentiating gives,
\begin{equation*} e^{- (1-\delta) \frac{\|x_i-x_j\|_2^2}{2\sigma}}\geq e^{\frac{-\| \Phi x_i-\Phi x_j \|_2^2}{2\sigma}} \geq e^{-(1+\delta) \frac{\|x_i-x_j\|_2^2}{2\sigma}},\end{equation*}
which can be rewritten as,
\begin{equation} W_{i,j} e^{\delta \frac{\|x_i-x_j\|_2^2}{2\sigma}}\geq \tilde{W}_{i,j} \geq W_{i,j}e^{-\delta \frac{\|x_i-x_j\|_2^2}{2\sigma}}.\label{eq:Weq1}\end{equation}
Hence,
\begin{equation*} \sum_{j} W_{i,j} e^{\delta \frac{\|x_i-x_j\|_2^2}{2\sigma}}\geq \sum_{j}\tilde{W}_{i,j} \geq \sum_{j}W_{i,j}e^{-\delta \frac{\|x_i-x_j\|_2^2}{2\sigma}}.\end{equation*}
Letting $C=\max_{i,j}\left\{\frac{\|x_i-x_j\|_2^2}{2\sigma}\right\}$,
\begin{equation*} D_{i,i} e^{\delta C}\geq \tilde{D}_{i,i} \geq D_{i,i}e^{-\delta C},\end{equation*}
and taking the square root gives,
\begin{equation} \sqrt{D_{i,i}} e^{\frac{\delta}{2} C}\geq \sqrt{\tilde{D}_{i,i} }\geq \sqrt{D_{i,i}}e^{-\frac{\delta}{2} C}.\label{eq:Weq2}\end{equation}
Dividing \eqref{eq:Weq1} by $\sqrt{\tilde{D}_{i,i}}\sqrt{\tilde{D}_{j,j}}$,
\begin{equation*} \frac{W_{i,j}}{\sqrt{\tilde{D}_{i,i}}\sqrt{\tilde{D}_{j,j}}} e^{\delta C}\geq \tilde{A}_{i,j} \geq \frac{W_{i,j}}{\sqrt{\tilde{D}_{i,i}}\sqrt{\tilde{D}_{j,j}}}e^{-\delta C},\end{equation*}
and by using inequality \eqref{eq:Weq2} gives,
\begin{equation*} \frac{W_{i,j}}{\sqrt{D_{i,i}}\sqrt{D_{j,j}}e^{ -\delta C}} e^{ \delta C}\geq \tilde{A}_{i,j} \geq \frac{W_{i,j}}{\sqrt{D_{i,i}}\sqrt{D_{j,j}}e^{ \delta  C}}e^{- \delta C}.\end{equation*}
Subtracting by $A_{i,j}$ we have,
\begin{equation*} A_{i,j}(e^{ 2 \delta C} -1)  \geq \tilde{A}_{i,j} -  A_{i,j}\geq A_{i,j}(e^{- 2\delta C}-1)\end{equation*}
Since $e^{ 2 \delta C}-1 \geq 1- e^{ -2 \delta C}$ for all $\delta>0$, and $A_{i,j}=\frac{W_{i,j}}{D_{i,j}^{-1/2}  D_{i,j}^{-1/2}}\leq 1$
\begin{equation}|\tilde{A}_{i,j} -  A_{i,j}|\leq e^{ 2 \delta C}-1.\label{eq:Aeq1}\end{equation}

To bound \eqref{eq:Aeq1} by $\epsilon$ requires $\delta$ to satisfy, 
\begin{equation*} \delta \leq \frac{\log ( \epsilon +1)}{2C} \ \ \ \ \ \forall i,j.  \end{equation*}

By simple calculation it can be shown that $ \frac{\log ( \epsilon +1)}{2C} \geq \frac{1}{4C}\epsilon  $ for $0<\epsilon<1 $.  
Thus, we need $\delta = \frac{\epsilon}{4 C}$,
so that
\begin{equation*} | \tilde{A}_{i,j} -  A_{i,j} | \leq \epsilon,\end{equation*}
 holds for all $i,j$.  
  \end{proof}

\begin{corollary}
Let $\Phi$ be a $m\times n$ Gaussian matrix and let $0<\epsilon <1$, $\delta = \frac{\epsilon}{4 \max_{i,j}\left\{ \frac{\|x_i - x_j\|_2^2}{2 \sigma}\right\}}$.  Then with high probability, $\frac{1}{\sqrt{m}}\Phi$ satisfies RIP  with parameters $(r,\delta)$  provided that the number of measurements, 
\[m=O(\frac{r}{\epsilon^2} \log \frac{n}{\epsilon^2r}).\]
\end{corollary}
\begin{proof}
Assume $x_i$ is $s$-sparse then $x_i-x_j$ is $2s$-sparse.  By the Gaussian measurement matrix Theorem in ~\cite{RV08} the number of measurements required for $\Phi$ to satisfy the RIP with parameters $(2s,\delta)$ with high probability is $m\geq \frac{c(t)2s}{\delta^2} \log \frac{n}{\delta^2 2s}$.  For $\delta = \frac{\epsilon}{4 \max_{i,j}\{ \frac{\|x_i - x_j\|_2^2}{2 \sigma}\}}$ this gives,
 \[m=O(\frac{2s}{(\frac{\epsilon}{4 C})^2} \log \frac{n}{(\frac{\epsilon}{4 C})^2 2s})=O(\frac{s}{\epsilon^2} \log \frac{n}{\epsilon^2 s}).\]
\end{proof}

  The entries of $\tilde{A}$ can be made arbitrarily close to the entries of $A$ when taking $m=O(\frac{s}{\epsilon^2} \log \frac{n}{\epsilon^2s})$ random Gaussian measurements.  Additional bounds on the number of measurements can be a found when using matrices with random Bernoulli entries or random rows of the Fourier transform ~\cite{RV08}.  Random Gaussian matrices are used here to demonstrate the order of the number of measurements required to achieve our bounds.  
 
 \begin{corollary}\label{cor:AnormCS}
 Let $A$ and $\tilde{A}$ be define as in Theorem~\ref{thm:AijCS}, then
 \[\| A-\tilde{A}\|_F \leq N  \epsilon.\]
  \end{corollary}
  \begin{proof}
  \begin{equation*}\| A-\tilde{A}\|_F =\sqrt{\sum_{i,j}(A_{i,j}-\tilde{A}_{i,j})^{2}}
  \leq \sqrt{\sum_{i,j}\epsilon^{2}} = \sqrt{N^2  \epsilon^2} 
  \end{equation*}
  \end{proof}

Theorem \ref{thm:VkCS} illustrates the robustness of the spectral clustering coordinates under small perturbations from taking compressed measurements. 

\begin{theorem}
Let $A$ and $\tilde{A}$ be defined as in Theorem~\ref{thm:AijCS} with $v(i)$ and $\tilde{v}(i)$ defined in Corollary~\ref{cor:Vk(i)}. Given there is a $\alpha>0$ such that $\lambda_k -\lambda_{k+1}\geq \alpha$ and $\lambda_k\geq \alpha$, if the number of measurements
\[m=O(\frac{r}{\epsilon^2} \log \frac{n}{\epsilon^2r}),\] then with high probability,
\[\|\tilde{v}(i)-v(i) Q\|_2 \leq (1+\sqrt{2})\frac{N  \epsilon}{\alpha}.\]
\label{thm:VkCS}

\end{theorem}
\begin{proof}
Assume there exists a $\alpha>0$ such that $\lambda_k -\lambda_{k+1}\geq \alpha$ and $\lambda_k\geq \alpha$.   Given an $\epsilon>0$, $| \tilde{A}_{i,j} -  A_{i,j} | \leq \epsilon$ by Theorem~\ref{thm:AijCS}.  This gives 
\[ \| A-\tilde{A}\|_F =\sqrt{\sum_{i,j}(A_{i,j}-\tilde{A}_{i,j})^{2}} \leq \sqrt{\sum_{i,j}\epsilon^{2}} = \sqrt{N^2  \epsilon^2} . \] 
With the eigengap $\lambda_k -\lambda_{k+1}\geq \alpha$ and the block decomposition of $V$ and $\tilde{V}$, Theorem~\ref{thm:VkQ} bounds the column space of $\{v_1,v_2,...,v_k\}$ and $\{\tilde{v}_1,\tilde{v}_2,...,\tilde{v}_k\}$ by \[\|\tilde{V}_k-V_k Q\|_2 \leq (1+\sqrt{2})\frac{1}{\alpha}\| A - \tilde{A} \|_F.\]  The spectral clustering coordinates \[\|\tilde{v}(i)-v(i) Q\|_2 \leq 1+\sqrt{2})\frac{N  \epsilon}{\alpha}\]
can then be bounded by Corollary~\ref{cor:Vk(i)}. Combining these gives the result. 
\end{proof}

\subsection{Perturbation of the entries of A from missing measurements}
The size and complexity of data grows exponentially with advancing technology.  Often data can be lost, noisy or corrupted or acquiring data could be costly to obtain, as in medical MRI acquisition.  Clustering algorithms can not be applied directly to incomplete data.   Given a fraction of the entries of a one wishes to recover the missing entries under the constraint that the unknown matrix is low rank.   This non-convex low rank minimization can be solved using nuclear-norm convex relaxation.   Matrix completion is the task of recovering an unknown matrix form a small subset of its entries.  This is possible for low rank matrices under some constraints on the matrix known as the strong incoherence property with $O(rank(X)\times N \log^2 N)$ samples via nuclear-norm convex optimization \cite{emcco09,MCWN}.

\begin{definition}
A matrix $X$ obeys the \textbf{strong incoherence property} with parameter $\ipp$ if the following hold.
\begin{enumerate}
	\item Let $P_Z$ (resp. $P_Y$) be the orthogonal projection onto the singular vectors $z_1,...,z_r$ (resp. $y_1,...,y_r$) of $X\in \mathbb{R}^{N \times n}$ of rank $r$.  For all pairs $(a,a')\in [N]\times [N]$ and $(b,b')\in [n]\times [n]$,
	\[ \left| \langle e_a,P_Z e_{a'}\rangle -\frac{r}{N}1_{a=a'} \right| \leq\ipp\frac{\sqrt{r}}{N},\]
	\[ \left| \langle e_b,P_Y e_{b'}\rangle -\frac{r}{n}1_{b=b'} \right|\leq\ipp\frac{\sqrt{r}}{n}.\]
	\item Let $\Xi$ be the "sign matrix" defined by
	\[\Xi=\sum_{i\in[r]}z_i y_i^*. \]
	For all $(a,b)\in[N]\times[n]$,
	\[ |\Xi_{a,b}|\leq \ipp \frac{\sqrt{r}}{\sqrt{Nn}}.\]
\end{enumerate}

\end{definition}

We will make use of the following theorem, which is a reformulation of Theorem 7 rigorously proved in \cite{MCWN}.

\begin{theorem}{\cite{MCWN}}
Let $X\in \mathbb{R}^{N \times n}$ be a fixed matrix of rank $r$ obeying the strong incoherence property with parameter $\ipp$. Suppose we observe a fraction $p=(\# \text{ of entries observed})/(N n)$ of entries of $X$ with locations sampled uniformly at random with noise $\|P_{\Omega}(X)-P_{\Omega}(\hat{X})\|_F\leq \delta $.  Then with high probability the solution to the matrix completion problem $\hat{X}$ satisfies, 
 \[\|X-\hat{X}\|_F\leq 4\sqrt{\frac{(2+p)\min(N,n)}{p}}\delta +2\delta.\] \label{thm:Xhat}
\end{theorem}
This provides a bound on the recovery error from matrix completion with noisy observations.  Using matrix completion with $O(r\times N \log^2 N)$ samples to define a local distance we show that $\hat{A}$ can be made arbitrarily close to $A$.

\begin{theorem}
Let $W_{i,j}=e^{ - \frac{\|x_i - x_j\|_2^2}{2 \sigma}}$ and $A= D^{-1/2} W D^{-1/2}$ where $D_{i,i}= \sum_{k=1}^N W(x_i,x_k)$.  And let $\hat{W}_{i,j}=e^{ - \frac{\| \hat{x}_i - \hat{x}_j\|_2^2}{2 \sigma}}$ and $\hat{A}=\hat{D}^{-\frac{1}{2}}\hat{W}\hat{D}^{-\frac{1}{2}}$ where $\hat{D}_{i,i}= \sum_{k=1}^N \hat{W}(x_i,x_k)$. 
Assume the data matrix $X$ obeys the strong incoherence property with parameter $\ipp$ under the same assumptions of Theorem~\ref{thm:Xhat}. If 
\[4\sqrt{\frac{(2+p)\min(N,n)}{p}}\delta +2\delta \leq \frac{\epsilon}{16 \max_{i,j} \left\{ \frac{\|x_i-x_j\|_2}{2\sigma} \right\}}\] then with high probability, 
 \begin{equation*} | \hat{A}_{i,j} -  A_{i,j} | \leq \epsilon.\end{equation*} 
\label{thm:AijMC}
\end{theorem}
\begin{proof}

From the performance guarantees for matrix completion problem with noise proved in Theorem 7 of \cite{MCWN} we have,
\begin{equation} 
\|X- \hat{X}\|_F \leq 4\sqrt{\frac{(2+p)\min(N,n)}{p}}\delta +2\delta \stackrel{def}{=} \gamma.
\label{eq:mcXFbound}
\end{equation}
\begin{equation*} 
\|X- \hat{X}\|_F^2 =\sum_{i=1}^N \|x_i-\hat{x}_i\|_2^2\leq  \gamma^2
\end{equation*}
implies \[\|x_i-\hat{x}_i\|_2\leq  \gamma.\]
Applying the triangle inequality gives,
\[\|\hat{x}_i-\hat{x}_j\|_2 \leq \|\hat{x}_i-x_j\|_2+\|x_i-x_j\|_2+\|x_i-\hat{x}_j\|_2 \leq \|x_i-x_j\|_2 + 2 \gamma.\] 
Similarly,
\[ \|x_i-x_j\|_2 \leq \|\hat{x}_i-\hat{x}_j\|_2 + 2 \gamma.\] 
Combining these gives,
\begin{equation}
 \|x_i-x_j\|_2 - 2 \gamma \leq \|\hat{x}_i-\hat{x}_j\|_2 \leq \|x_i-x_j\|_2 + 2 \gamma.
\label{eq:mcxx2bound}
\end{equation}
Squaring gives,
\[ \|x_i-x_j\|_2^2 - 2 \gamma\|x_i-x_j\|_2  +  4 \gamma^2 \leq \|\hat{x}_i-\hat{x}_j\|_2 \leq \|x_i-x_j\|_2^2 + 2 \gamma\|x_i-x_j\|_2  +  4 \gamma^2.\]
Multiplying by $\frac{-1}{2\sigma}$ and exponentiating gives,
\begin{equation*} e^{\frac{4\gamma \|x_i-x_j\|_2-4\gamma^2}{2\sigma}} e^{\frac{- \|x_i-x_j\|_2^2}{2\sigma}}\geq e^{\frac{-\| \hat{x}_i-\hat{x}_j \|_2^2}{2\sigma}} \geq e^{\frac{-4\gamma \|x_i-x_j\|_2-4\gamma^2}{2\sigma}} e^{ \frac{-\|x_i-x_j\|_2^2}{2\sigma}},\end{equation*}
which can be rewritten as,
\begin{equation} W_{i,j} e^{\frac{4\gamma \|x_i-x_j\|_2-4\gamma^2}{2\sigma}}\geq \hat{W}_{i,j} \geq W_{i,j}e^{\frac{-4\gamma \|x_i-x_j\|_2-4\gamma^2}{2\sigma}}.\label{eq:Whateq1}\end{equation}
Hence,
\begin{equation*} \sum_{j} W_{i,j} e^{\frac{4\gamma \|x_i-x_j\|_2-4\gamma^2}{2\sigma}}\geq \sum_{j} \hat{W}_{i,j} \geq  \sum_{j} W_{i,j}e^{\frac{-4\gamma \|x_i-x_j\|_2-4\gamma^2}{2\sigma}}.\end{equation*}
Letting $C=4\max_{i,j}\left\{\frac{\|x_i-x_j\|_2}{2 \sigma}\right\}$,
\begin{equation*} D_{i,i} e^{\gamma C -\frac{4\gamma^2}{2\sigma}}\geq \hat{D}_{i,i} \geq D_{i,i}e^{-\gamma C - \frac{4\gamma^2}{2\sigma}},\end{equation*}
and taking the square root gives,
\begin{equation}
\sqrt{D_{i,i}} e^{\frac{\gamma C}{2}-\frac{4\gamma^2}{4\sigma}}\geq \sqrt{\hat{D}_{i,i}} \geq \sqrt{D_{i,i}}e^{\frac{-\gamma C}{2} -\frac{4\gamma^2}{4\sigma}}
.\label{eq:Dhateq2}\end{equation}
Dividing \eqref{eq:Whateq1} by $\sqrt{\hat{D}_{i,i}}\sqrt{\hat{D}_{j,j}}$,
\begin{equation*} \frac{W_{i,j}}{\sqrt{\hat{D}_{i,i}}\sqrt{\hat{D}_{j,j}}} e^{\gamma C-\frac{4\gamma^2}{2\sigma}} \geq \hat{A}_{i,j} \geq \frac{W_{i,j}}{\sqrt{\hat{D}_{i,i}}\sqrt{\hat{D}_{j,j}}}e^{-\gamma C-\frac{4\gamma^2}{2\sigma}} ,\end{equation*}
and by using inequality \eqref{eq:Dhateq2} gives,
\begin{equation*} \frac{W_{i,j}}{\sqrt{D_{i,i}}\sqrt{D_{j,j}}e^{-\gamma C}} e^{\gamma C} \geq \tilde{A}_{i,j} \geq \frac{W_{i,j}}{\sqrt{D_{i,i}}\sqrt{D_{j,j}}e^{\gamma C}}e^{-\gamma C}.\end{equation*}
Subtracting by $A_{i,j}$ we have,
\begin{equation*} A_{i,j}(e^{2 \gamma C} -1)  \geq \hat{A}_{i,j} -  A_{i,j}\geq A_{i,j}(e^{-2 \gamma C}-1)\end{equation*}
Since $e^{2 \gamma C}-1 \geq 1- e^{-2 \gamma C}$ for all $\gamma>0$, and $A_{i,j}=\frac{W_{i,j}}{D_{i,j}^{-1/2}  D_{i,j}^{-1/2}}\leq 1$
\begin{equation}|\hat{A}_{i,j} -  A_{i,j}|\leq e^{2\gamma C}-1.\label{eq:Aeq2}\end{equation}

To bound \eqref{eq:Aeq2} by $\epsilon$ requires $\gamma$ to satisfy, 
\begin{equation*} \gamma \leq \frac{\log ( \epsilon +1)}{2C}.  \end{equation*}

By simple calculation it can be shown that $ \frac{\log ( \epsilon +1)}{2C} \geq \frac{1}{4C}\epsilon  $ for $0<\epsilon<1 $.  
Thus, we need $\gamma = \frac{\epsilon}{4 C}$,
so that
\begin{equation*} | \tilde{A}_{i,j} -  A_{i,j} | \leq \epsilon,\end{equation*}
 holds for all $i,j$.  
  \end{proof}

\begin{theorem}
Let $A$ and $\tilde{A}$ be defined as in Theorem~\ref{thm:AijMC} with $v(i)$ and $\tilde{v}(i)$ defined in Corollary~\ref{cor:Vk(i)}. Given there is a $\alpha>0$ such that $\lambda_k -\lambda_{k+1}\geq \alpha$ and $\lambda_k\geq \alpha$, if 
\[4\sqrt{\frac{(2+p)\min(N,n)}{p}}\delta +2\delta \leq \frac{\epsilon}{16 \max_{i,j} \left\{ \frac{\|x_i-x_j\|_2}{2\sigma} \right\}}\] then with high probability
\[\|\tilde{v}(i)-v(i) Q\|_2 \leq (1+\sqrt{2})\frac{N  \epsilon}{\alpha}.\]
\label{thm:VkMC}
\end{theorem}
\begin{proof}
The proof is similar to that of Theorem~\ref{thm:VkCS} by combining Theorem~\ref{thm:AijMC} and Corollary~\ref{cor:Vk(i)}. 
\end{proof}

\section{Numerical Results}\label{sec:res}
The experiments in this section were preformed on three grayscale image sets, a synthetic set, a set of face images and a set of handwritten digits. The spectral clustering method requires local distances between data points.   Each data point is an image and the distances between images, $d(x_i, x_j)$, are as defined above using Frobenius norm distance \cite{NLA} and compressed sensing measurements.

	\begin{figure}[h!]
\begin{center}
\includegraphics[width=\textwidth]{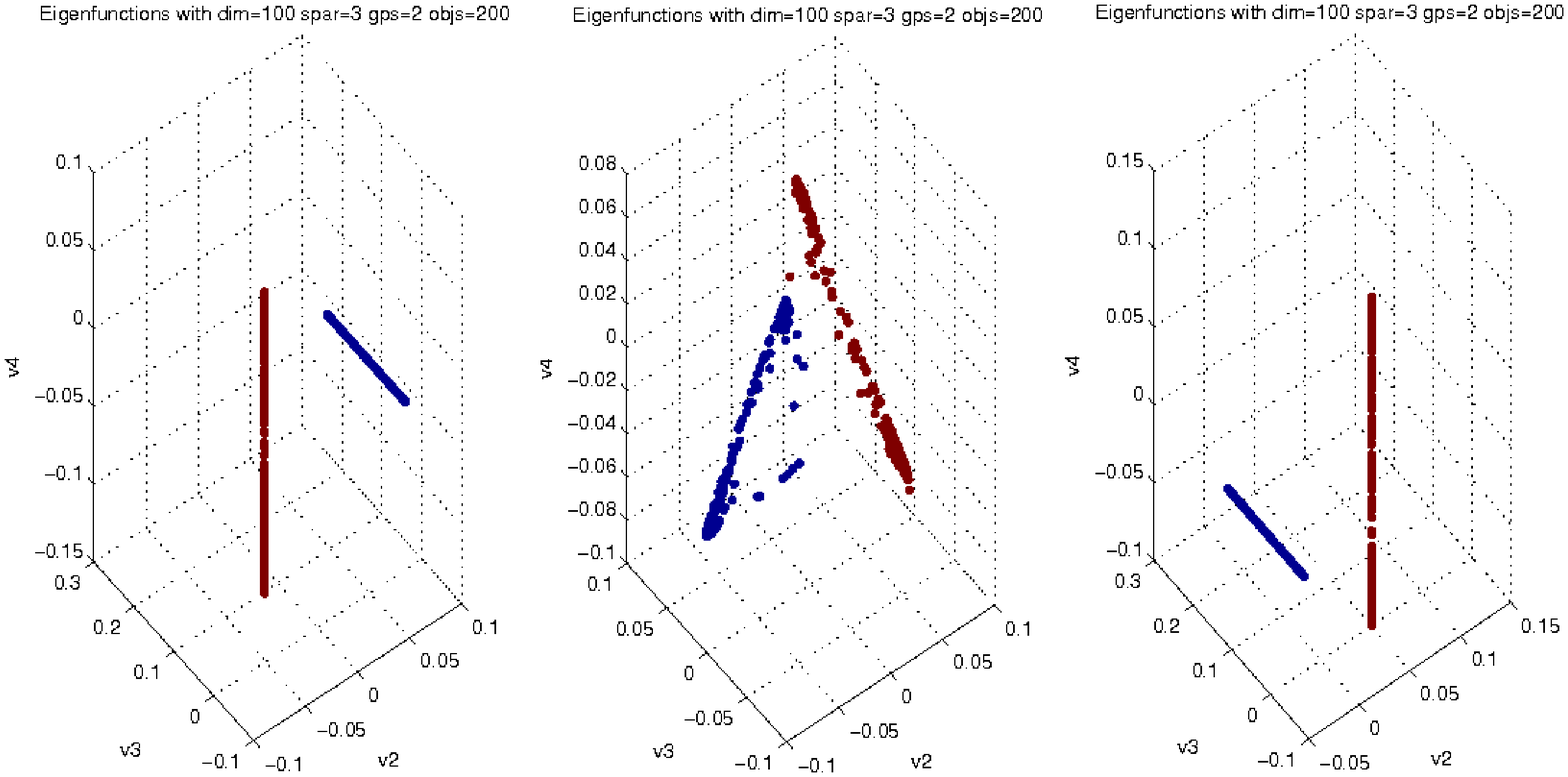}  
\caption {\footnotesize   Data set of synthetic 100 dimensional 3-sparse image vector with two underlying clusters, each image is projected onto the first three nontrivial eigenvectors.   \textbf{Left:} Using the full $100$ dimensional image vector (standard spectral clustering) \textbf{Center:} Using 3 compressed sensing measurements of each image.  \textbf{Right:} Using 30 compressed sensing measurements. }\label{fig:synt100eig}
\end{center}
\end{figure}
	\subsection{Synthetic - Advertisers clustered by keywords}
	This first experiment is clustering advertisers and keywords for a search engine query.  Each advertiser $ x_i $ only pays at most $r$ keywords that they want to be connected to their product, making $x_i \in \mathbb{R}^n$, $s$-sparse in $n$ keywords feature space.  The goal is to groups the advertisers such that advertisers within the same group have pay for similar keywords to identify a wider range of targeted keywords to advertisers.  Spectral clustering provides a balance of cluster compactness, conductance and proportional size.  
Here synthetic feature vectors are created from an unknown unitary sparse transform with a varying sparsity level.  Figure~\ref{fig:synt100eig} shows the clustering of a data set with two hidden data clouds.  The left Figure shows clustering using only 10 random measurements and the middle using 30, compared to standard techniques using the full image on the right.  Both methods provide a linear separable dimension reduction but our compressive method is only using three measurements compared to requiring the full dimension.  The number of measurements used here is taken much lower than our theoretical requirement yet it performs at the same perfect classification rate.
	
	\begin{figure}[h!]
	\begin{center}
		\includegraphics[width=\textwidth]{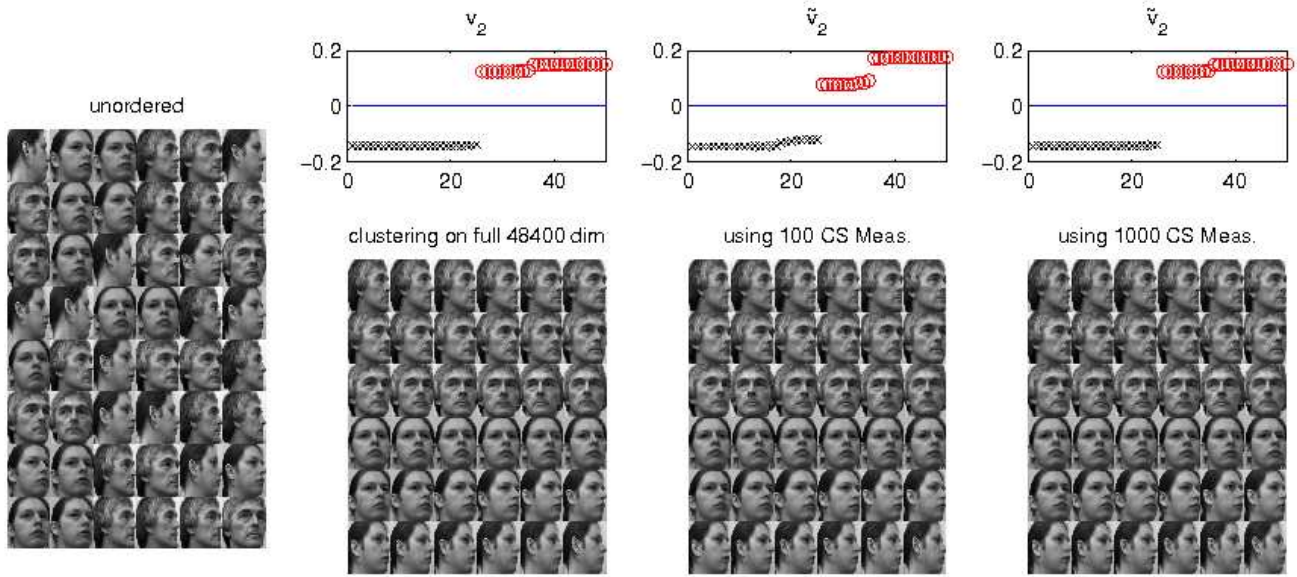} \\
		 \caption {\footnotesize 50 - $220 \times 220$ images of two different faces in a range of poses from profile to frontal views.   \textbf{Far Left:} Unordered. \textbf{Left Center:} Ordered by the second eigenvector from standard spectral clustering. \textbf{Right Center:} By compressive spectral clustering using 100 measurements.  \textbf{Far Right:} Using 1000 measurements.}\label{fig:face_or_eig}
	\end{center}
\end{figure}
	
	\subsection{Face Database}
	This experiment uses a database of 36 112x92 pixel images of the same person, rotating his head, from the UMIST Face Database \cite{FRFUV98}.    It is well known that images have a sparse Fourier and wavelet transform, a fact used in JPEG and JPEG2000 image compression, but the optimal sparsifying transform here is unknown.  The ideal transform would only capture the desired differences in a given set. Here it would transform the image vector into a signal of sparsity level proportional to the degrees of freedom of rotation and/or the number of subjects used.  In Figure~\ref{fig:face_or_eig} the ambient dimension is 10,000 and only 10 measurements are used.  The second eigenvector of $A$ and $\tilde{A}$ both capture the underlying rotation from their shuffled order but our method uses $1/1000$ the number of measurements.  

	  Often a signal may not be obtainable, but random measurements of the signal might be.  For example, images where only the inner products with random rows of the Fourier transform are observed.   We show that with $m$ random measurements, a set of images can be clustered by their measurements in the same way as if the clustering was performed on the full images themselves.  
	  
	   \begin{figure}[!h]
	\begin{center}
		\includegraphics[width=.5\textwidth]{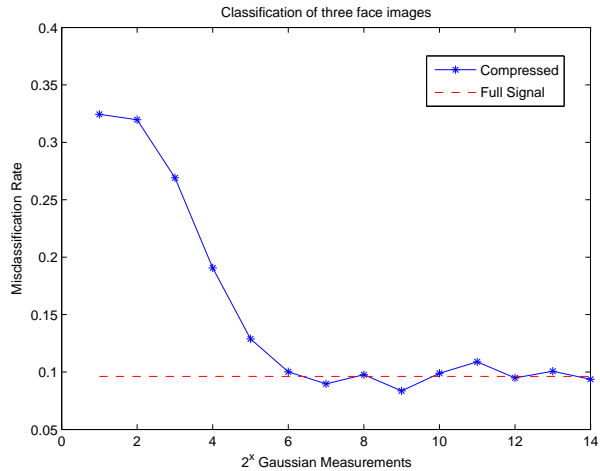}
	\end{center}
	\caption {\footnotesize A dataset of 100 images of three different people's faces in a range of poses from profile to frontal views.  The images are classified by applying k-means to the compressive spectral clustering coordinates ($2^{nd}$ and $3^{rd}$ eigenvectors), using a range of measurements.  The average misclassification rates of 100 trials are plotted against the average misclassification rate resulting from standard spectral clustering applied to the full images.  } \label{fig:misclass3facesCS}
\end{figure}

	  In comparing our compressed results with standard spectral clustering we achieve the same classification rate with fewer measurements than full dimension of the images.  A misclassification rate of .1 is achieved using only $2^8$ measurements where as the full $2^{13}$ dimensional signal is required in standard spectral clustering, as seen in Figure~\ref{fig:misclass3facesCS}.  Here the clustering was performed on a set of images of three different people.  There is more than one way to cluster a set of more than two faces.   In addition to grouping person A, B and C, the images could be grouped by skin tones, hair color or male/female, all of which present a valid classification.  Our method is only guaranteed to perform as well as standard spectral clustering.  Because spectral clustering is unsupervised, given no labeled examples as inputs, the natural groupings found by spectral clustering may fail to match a desired labeling causing a higher misclassification rate.

	\subsection{Handwritten Digits}
	The sparsifying transform here is again unknown.  The ideal transform would transform the image vector into a signal of sparsity level proportional to the number of digits $\{0,1,...,9\}$ to be classified.  Figures \ref{fig:dig_1_3_eig_sort} and \ref{fig:dig137eig} show the clustering of pairs on handwritten digits, using standard spectral clustering on the left and compressive spectral clustering on the right.  The structure for the column space of $\{ v_1,v_2,v_3,v_4\}$ is maintained under the perturbation from making compressed sensing measurements.  The clusters defined by the first eigenvectors of $A$ from using the true Euclidean distances is equal to the clusters defined by the first eigenvectors of $\tilde{A}$ from taking compressed measurements.

 \begin{figure}[!h]
	\begin{center}
		\includegraphics[width=\textwidth]{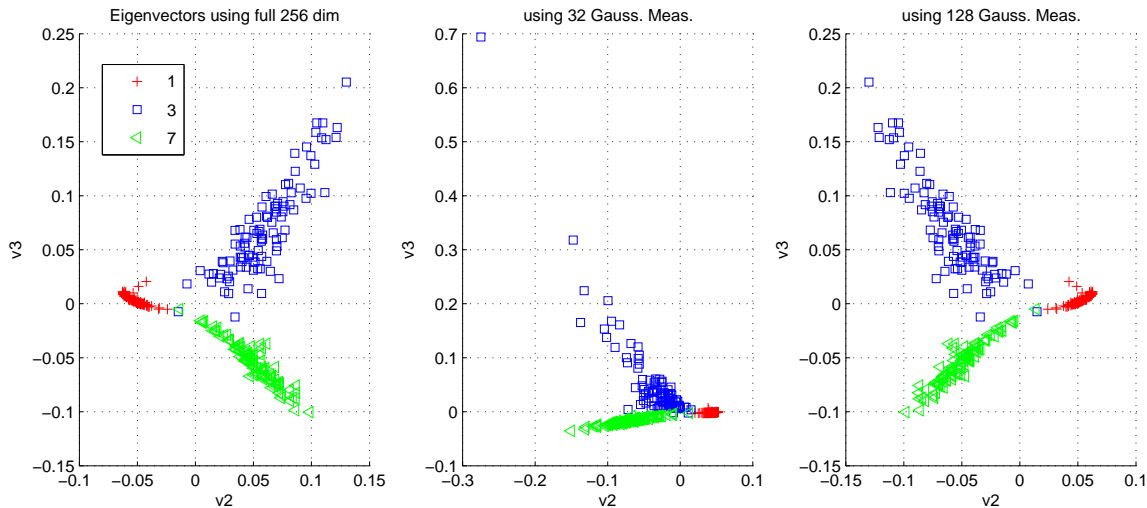}
	\end{center}
	\caption {\footnotesize   The handwritten digits $\{1,3,7\}$ data set is projected onto the 2nd and 3rd eigenvectors of the graph formed \textbf{Left:} using Euclidean distances between points, \textbf{Middle:} using distances from 32 Gaussian measurements, \textbf{Right:} using distances from 128 random Gaussian measurements. } \label{fig:dig137eig}
\end{figure}
	  
 \subsection{Synthetic Images with missing entries}
 Here we use a data set of 1000 synthetic images of three different classes where we can control the rank of the data matrix.  In both experiments only 10\% of the entries are observed.  The first experiment varies the approximate rank of the data where the images are stacked as row vectors.  As the rank of the matrix is artificially increased the number of underlying clusters is blurred.  Figure~\ref{fig:misclass3facesMCandCS}  that as the approximate rank increases so does the misclassification rate.  The second experiment uses a combination of the two perturbation errors, from matrix completion and compressed measurements.  Figure~\ref{fig:misclass3facesMCandCS} shows that as the number of measurements increase the span of the compressed spectral coordinates with matrix completion is close to the traditional spectral coordinates.

\begin{figure}[!h]
	\begin{center}
		\includegraphics[width=0.5\textwidth]{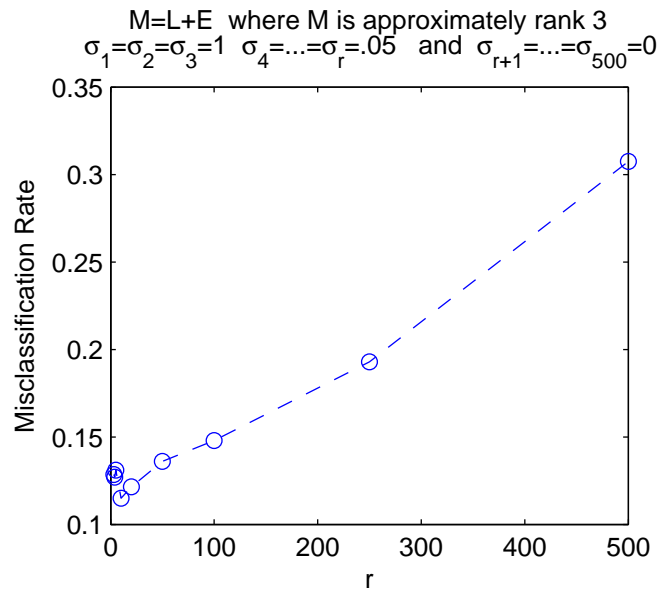}
	\end{center}
	\caption {\footnotesize $M$ is a rank 3 matrix with 1000 images (three clusters) of dimension 500 stacked as rows with only 10\% of the entries observed.  The images are classified by applying the matrix completion then using the first three spectral clustering coordinates ($1^{st}$, $2^{nd}$ and $3^{rd}$ eigenvectors). The number of clusters is pertubed by increasing the rank of $M$. } \label{fig:misclass3facesMC}
\end{figure}

\begin{figure}[!h]
	\begin{center}
		\includegraphics[width=0.5\textwidth]{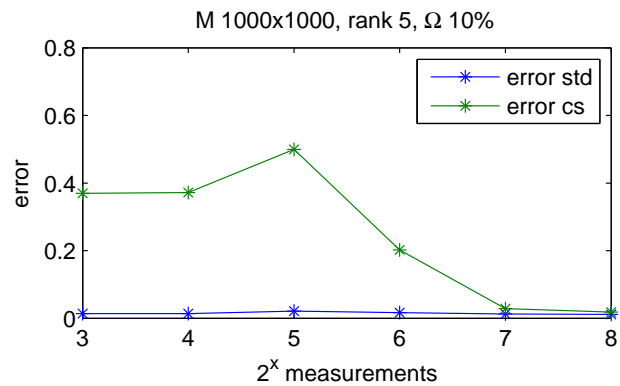}
	\end{center}
	\caption {\footnotesize 1000, 100-sparse synthetic images of dimension 1000 with only 10\% of the entries observed.  The error$=\|\tilde{V}_3 - V_3 Q \|_2$ between the span of the first three eigenvectors of $A$ and $\tilde{A}$ is plotted versus the number of compressed measurements taken.  The images are classified by applying the matrix completion and using compressed spectral clustering coordinates ($1^{st}$, $2^{nd}$ and $3^{rd}$ eigenvectors).  } \label{fig:misclass3facesMCandCS}
\end{figure}



%

%

%
%
%
%

\ifCLASSOPTIONcaptionsoff
  \newpage
\fi



\bibliographystyle{IEEEtran}
%

\bibliography{HunterStrohmerSpectralClusteringArXix}

\begin{thebibliography}{10}
\providecommand{\url}[1]{#1}
\csname url@samestyle\endcsname
\providecommand{\newblock}{\relax}
\providecommand{\bibinfo}[2]{#2}
\providecommand{\BIBentrySTDinterwordspacing}{\spaceskip=0pt\relax}
\providecommand{\BIBentryALTinterwordstretchfactor}{4}
\providecommand{\BIBentryALTinterwordspacing}{\spaceskip=\fontdimen2\font plus
\BIBentryALTinterwordstretchfactor\fontdimen3\font minus
  \fontdimen4\font\relax}
\providecommand{\BIBforeignlanguage}[2]{{%
\expandafter\ifx\csname l@#1\endcsname\relax
\typeout{** WARNING: IEEEtran.bst: No hyphenation pattern has been}%
\typeout{** loaded for the language `#1'. Using the pattern for}%
\typeout{** the default language instead.}%
\else
\language=\csname l@#1\endcsname
\fi
#2}}
\providecommand{\BIBdecl}{\relax}
\BIBdecl

\bibitem{ESL01}
T.~Hastie, R.~Tibshirani, and J.~Friedman, \emph{The Elements of Statistical
  Learning}, ser. Springer Series in Statistics.\hskip 1em plus 0.5em minus
  0.4em\relax New York, NY, USA: Springer New York Inc., 2001.

\bibitem{NCIS}
J.~Shi and J.~Malik, ``Normalized cuts and image segmentation,'' \emph{IEEE
  Transactions on Pattern Analysis and Machine Intelligence}, vol.~22, pp.
  888--905, 1997.

\bibitem{CLSDR}
R.~Calderbank, S.~Jafarpour, and R.~Schapire, ``Compressed learning: Universal
  sparse dimensionality reduction and learning in the measurement domain,''
  \emph{Manuscript}, 2009.

\bibitem{DDCS}
D.~L. Donoho, ``Compressed sensing,'' \emph{IEEE Trans. Inform. Theory}, vol.
  52(4), pp. 1289--1306, 2006.

\bibitem{SSRIIM}
E.~Cand\`es, J.~Romberg, and T.~Tao, ``Stable signal recovery from incomplete
  and inaccurate measurements,'' \emph{Communications on Pure and Applied
  Mathematics}, vol.~59, no.~8, pp. 1207--1223, 2006.

\bibitem{emcco09}
E.~Cand\`es and B.~Recht, ``{Exact matrix completion via convex
  optimization},'' \emph{Foundations of Computational Mathematics}, vol.~9,
  no.~6, pp. 717--772, 2009.

\bibitem{pcrnomc09}
E.~Cand\`es and T.~Tao, ``{The power of convex relaxation: Near-optimal matrix
  completion},'' \emph{Information Theory, IEEE Transactions on}, vol.~56,
  no.~5, pp. 2053--2080, 2010.

\bibitem{KKSC}
I.~Dhillon, Y.~Guan, and B.~Kulis, ``{Kernel k-means: spectral clustering and
  normalized cuts},'' in \emph{Proceedings of the tenth ACM SIGKDD
  international conference on Knowledge discovery and data mining}.\hskip 1em
  plus 0.5em minus 0.4em\relax ACM, 2004, pp. 551--556.

\bibitem{MCSC}
S.~Y. Jianbo, S.~X. Yu, and J.~Shi, ``Multiclass spectral clustering,'' in
  \emph{In International Conference on Computer Vision}, 2003, pp. 313--319.

\bibitem{SR4KM}
H.~Zha, X.~He, C.~Ding, H.~Simon, and M.~Gu, ``Spectral relaxation for k-means
  clustering.''\hskip 1em plus 0.5em minus 0.4em\relax MIT Press, 2001, pp.
  1057--1064.

\bibitem{FASC}
D.~Yan, L.~Huang, and M.~Jordan, ``{Fast approximate spectral clustering},'' in
  \emph{Proceedings of the 15th ACM SIGKDD international conference on
  Knowledge discovery and data mining}.\hskip 1em plus 0.5em minus 0.4em\relax
  ACM, 2009, pp. 907--916.

\bibitem{scwpd}
L.~Huang, D.~Yan, M.~I. Jordan, and N.~Taft, ``Spectral clustering with
  perturbed data,'' in \emph{Advances in Neural Information Processing Systems
  (NIPS)}, 2008.

\bibitem{IMC}
G.~W. Stewart, \emph{Introduction to Matrix Computation}.\hskip 1em plus 0.5em
  minus 0.4em\relax Academic Press, 1973.

\bibitem{REBP70}
C.~Davis and W.~M. Kahan, ``The rotation of eigenvectors by a perturbation.
  {III},'' \emph{SIAM J. Appl. Math.}, vol.~7, no.~1, 1970.

\bibitem{MPTSS}
G.~W. Stewart and J.~Sun, \emph{Matrix Perturbation Theory}.\hskip 1em plus
  0.5em minus 0.4em\relax Academic Press, Boston, 1990.

\bibitem{mcgvl}
G.~H. Golub and C.~F.~V. Loan, \emph{Matrix computations (3rd ed.)}.\hskip 1em
  plus 0.5em minus 0.4em\relax Baltimore, MD, USA: Johns Hopkins University
  Press, 1996.

\bibitem{OSCAA}
A.~Y. Ng, M.~I. Jordan, and Y.~Weiss, ``On spectral clustering: Analysis and an
  algorithm,'' \emph{Advances in Neural Information Processing Systems 14}, pp.
  849--856, 2001.

\bibitem{CTDLP05}
E.~J. Cand\`es and T.~Tao, ``Decoding by linear programming,'' \emph{IEEE
  Trans. Inform. Theory}, vol.~51, pp. 4203--4215, 2005.

\bibitem{RV08}
M.~Rudelson and R.~Vershynin, ``On sparse reconstruction from {F}ourier and
  {G}aussian measurements,'' \emph{Comm. Pure Appl. Math.}, vol.~61, pp.
  1025--1045, 2008.

\bibitem{MCWN}
E.~Cand\'es and Y.~Plan, ``{Matrix completion with noise},'' \emph{Proceedings
  of the IEEE}, vol.~98, no.~6, pp. 925--936, 2010.

\bibitem{NLA}
L.~N. Trefethen and D.~B. III, \emph{Numerical Linear Algebra}.\hskip 1em plus
  0.5em minus 0.4em\relax SIAM, 1997.

\bibitem{FRFUV98}
D.~B. Graham and N.~M. Allinson, ``Face recognition from unfamiliar views:
  Subspace methods and pose dependency,'' in \emph{FG}, 1998, pp. 348--353.

\end{thebibliography}

\begin{IEEEbiographynophoto}{Blake Hunter}
Blake is a Ph.D. candidate in applied mathematics at University of California, Davis, studying data mining via harmonic analysis.  His areas of research interest are applied harmonic analysis, image processing, diffusion maps, machine learning, data mining, and high-dimensional data analysis.
\end{IEEEbiographynophoto}

\begin{IEEEbiographynophoto}{Thomas Strohmer}
Thomas Strohmer received his M.S. and Ph.D. in Mathematics in 1991 and 1994
respectively from the University of Vienna, Austria. He was a Research Assistant at the Department of Mathematics, University of Vienna from 1991 to 1997. He spent one year as Erwin-Schroedinger fellow at the Department of Statistics at the Stanford University and then joined the Department of Mathematics at the University of California in Davis in 1998, where he is now Full Professor. He was a Visiting Professor at Stanford University, the Technical University of Denmark, and the Heinrich-Hertz Institute in Berlin. His general research interests are in harmonic analysis, numerical analysis, digital signal processing, and information theory. He is co-editor of two books and on the editorial board of several journals. He also serves as consultant to the signal processing and telecommunications industry.
\end{IEEEbiographynophoto}






\end{document}